\documentclass{amsart}

\usepackage{amsmath}
\usepackage{amssymb}
\usepackage{amsthm}
\usepackage{amsfonts}
\usepackage{amsrefs}
\usepackage{setspace}
\usepackage{mathtools}
\usepackage[T1]{fontenc}
\usepackage{verbatim}
\usepackage{mathrsfs}
\usepackage{indentfirst}
\usepackage{etoolbox}
\usepackage{IEEEtrantools}

\theoremstyle{plain}
\newtheorem{prop}{Proposition}[section]

\theoremstyle{definition}

\theoremstyle{plain}
\newtheorem{Metacommutation Action}{Theorem}[section]
\newtheorem{Sign of Permutation}{Theorem}[section]
\newtheorem{Fixed Points}[Sign of Permutation]{Theorem}
\newtheorem{Cycles}[Sign of Permutation]{Theorem}
\numberwithin{equation}{section}

\newtheorem{corollary}{Corollary}[section]

\theoremstyle{remark}

\theoremstyle{definition}
\newtheorem{Remark 1}{Remark}[section]
\newtheorem{Remark 2}[Remark 1]{Remark}

\makeatletter
\def\th@plain{%
  \thm@notefont{}
  \itshape 
}
\def\th@definition{%
  \thm@notefont{}
  \normalfont 
}
\makeatother

\newcommand{\overbar}[1]{\mkern 1.5mu\overline{\mkern-1.5mu#1\mkern-1.5mu}\mkern 1.5mu}
\makeatletter
\g@addto@macro\bfseries{\boldmath}
\makeatother

\title{Metacommutation as a Group Action on the Projective Line Over $\mathbb{F}_p$}
\author{Adam Forsyth}
\address{Georgetown Day School \\ 4200 Davenport Street NW, Washington, DC 20015}
\email{adambforsyth@gmail.com}

\author{Jacob Gurev}
\address{Mira Loma High School, 4000 Edison Avenue, Sacramento, CA 95821}
\email{jackgurev@gmail.com}

\author{Shakthi Shrima}
\address{Shrima Academy Homeschool, 17 Yellowtail Cove, Austin, TX 78745}
\email{s.s.shrima@gmail.com}

\date{}

\begin{document}
\maketitle

\begin{abstract}
Cohn and Kumar showed the quadratic character of $q$ modulo $p$ gives the sign of the permutation of Hurwitz primes of norm $p$ induced by the Hurwitz primes of norm $q$ under metacommutation. We demonstrate that these permutations are equivalent to those induced by the right standard action of $\operatorname{PGL}_2 (\mathbb{F}_p)$ on $\mathbb{P}^1 (\mathbb{F}_p)$. This equivalence provides simpler proofs of the results of Cohn and Kumar and characterizes the cycle structure of the aforementioned permutations. Our methods are general enough to extend to all orders over the quaternions with a division algorithm for primes of a given norm $p$. 
\end{abstract}

\section{Introduction}
\noindent Consider the Hamilton quaternions $\mathbb{H} = \mathbb{R} + \mathbb{R}i + \mathbb{R}j + \mathbb{R}k$, where $i, j,$ and $k$ are uniquely defined by the relation \[ i^2=j^2=k^2=ijk=-1. \]
In this paper, we study the subring of Hurwitz quaternions:
\[ \mathcal{H}= \mathbb{Z}\left[i, j, \frac{1}{2}(1+i+j+k)\right]. \]
(In words, this is the subring of quaternions whose components are either all elements of $\mathbb{Z}$ or all elements of $\mathbb{Z} + 1/2$.)

We begin by reviewing some fundamental definitions and results. For any quaternion
\[ h = a + bi + cj + dk, \]
we call
\begin{IEEEeqnarray*}{rCl}
h^\sigma &=& a - bi - cj - dk,\\
N(h) &=& a^2 + b^2 + c^2 + d^2, \\
\operatorname{tr}(h) &=& h + h^\sigma = 2a
\end{IEEEeqnarray*}
its \textit{conjugate}, \textit{norm}, and \textit{trace}, respectively.
The reader can easily verify that the norm is multiplicative. 

Because the ring of Hurwitz quaternions has left and right division algorithms, all of its ideals are principal. We call a Hurwitz integer $P$ \emph{prime} if it is irreducible. A Hurwitz integer is prime in $\mathcal{H}$ if and only if its norm is prime in $\mathbb{Z}$ (see \cite{csmith} for details.) In \cite{csmith}, Conway and Smith  prove any factorization of a  non-zero Hurwitz integer into Hurwitz primes is unique up to three phenomena, which we define below. (\cite{cperng} provides another exposition of this theorem.)

The Euclidean algorithm in $\mathcal{H}$ lets us factor any non-zero $h \in \mathcal{H}$ with norm $N(h) = p_1 p_2 \cdots p_n$, where the $p_i$ are prime, into a product of Hurwitz primes:
\[ h = P_1P_2 \cdots P_n, \quad \text{where } N(P_i) = p_i. \]
As in \cite{cohnkum}, we say this factorization of $h$ is \emph{modelled} on the factorization of $N(h)= p_1p_2 \cdots p_n$. When $N(h)$ is square-free, we call such a factorization unique up to \emph{unit migration}, because if the $p_i$ are distinct, then every factorization of $h$ modelled on 
\[ N(h) = \prod_{i=1}^n{p_i} \]
is of the form 
\begin{equation*}
h= (P_1u_1)(u^{-1}_1P_2u_2)\cdots(u^{-1}_{n-1}P_n), 
\end{equation*}
where the $u_i$ are units in $\mathcal{H}$. If a factorization of $h$ contains both a prime $P$ and $P^{\sigma}$ in sequence, we can replace $PP^{\sigma}$ by $P_1P_1^{\sigma}$ where $N(P)=N(P_1)$. We call such a factorization unique up to \emph{recombination}. As in \cite{cohnkum}, we say a left ideal $\mathcal{H}P$ for prime $P \in \mathcal{H}$ \emph{lies over} a rational prime $p$ if $N(P)=p$. 

The next proposition, which \cite{cperng} proves, establishes the main phenomenon we study in this paper.

\begin{prop}[Conway and Smith]
If $P$ and $Q$ are distinct primes in $\mathcal{H}$ that lie over rational primes $p$ and $q$ repectively, then $PQ$ has a factorization $Q'P'$ modelled on $qp$ that is unique up to unit migration.
\end{prop}

We call the process of ``swapping'' adjacent primes \emph{metacommutation}. Given $Q \in \mathcal{H}$ with norm coprime to $p$, every $P$ with norm $p$ has an associated prime $P'$ also of norm $p$, which satisfies
\[ PQ = Q'P' \quad \text{for some } Q' \in \mathcal{H}. \]
The \textit{metacommutation map} on primes of norm $p$ sends every $P$ to $P'$. For each $Q$, this map induces a permutation of the Hurwitz primes of norm $p$. In \cite{cohnkum}, Cohn and Kumar prove that the sign of this permutation is $\genfrac(){}{}{q}{p}$, where $\genfrac(){}{}{\cdot}{p}$ is the Legendre symbol, and that this permutation has
\[ 1 + \left(\frac{\operatorname{tr}(Q)^2 - 4q}{p}\right) \]
fixed points. Although the current literature gives several properties of the permutation the metacommutation map induces, it does not directly describe this permutation. 

In Sections 2 and 3 of this paper, we establish isomorphisms and bijections which simplify studying the metacommutation phenomenon; in Section 4, we show the metacommutation mapping is isomorphic to the right standard action of $\operatorname{PGL}_2(\mathbb{F}_p)$ on the projective line over $\mathbb{F}_p$; in Section 5, we provide new, shorter proofs of the results of Cohn and Kumar, and characterize much of the cycle strture of the permutations. 

\section{Definitions and Preliminaries}
\noindent Let $\mathcal{H}/\mathcal{H}p= \overline{\mathcal{H}}$, where $p$ is a fixed rational prime, and let $\overline P$ denote the reduction of any $P \in \mathcal{H}$ modulo $p$. Given a Hurwitz prime $P$ of norm $p$, we can then define the left ideal $\overline{\mathcal{H} P}= \left\{h\overline{P} \mid h \in \overline{\mathcal{H}}\right \}$, which is the modulo $p$ reduction of the ideal $\mathcal{H}P$. Note that two primes $P$ and $P'$ with norm p induce the same ideal in $\overline{\mathcal{H}}$ when they are left associates; this motivates us to consider uniqueness of the primes with norm $p$ only up to left multiplication by units.

All such ideals have dimension $2$. Indeed, since conjugation is an automorphism of $\mathcal{H}$, $\overline{\mathcal{H}P}$ has the same dimension as $\overline{\mathcal{H}P^\sigma}$. Because the elements $h \in \overline{\mathcal{H}P^\sigma}$ are exactly those where $hP$ equals $0$ inside of $\overline{\mathcal{H}}$, 
\[ \operatorname{dim}(\overline{\mathcal{H}P})+\operatorname{dim}(\overline{\mathcal{H}P^\sigma})=\operatorname{dim}(\overline{\mathcal{H}})=4, \]
and $\operatorname{dim}(\overline{\mathcal{H}P})=2$ as desired.

\cite{cperng} gives a one-to-one correspondence between primes in $\mathcal{H}$ and the two-dimensional left ideals of $\overline{\mathcal{H}}$. This lets us count the number of Hurwitz primes of norm $p$ through a bijective correspondence with points on the conic \begin{gather*}
C_p = \left\{(x, y, z) \in \mathbb{P}^2(\mathbb{F}_p) \mid x^2 +y^2 +z^2 =0 \right\}.
\end{gather*} 
The following proposition establishes this correspondence. (The proof we give follows the argument in \cite{cohnkum}.)
\begin{prop} There is a bijective correspondence between points on the conic $C_p$ and Hurwitz primes $P$ of norm $p$. 
\end{prop}
\begin{proof}
The trace function on the ideal $\overline{\mathcal{H}P}$ is linear. Being a linear function from a two-dimensional vector space over $\mathbb{F}_p$ to a one-dimensional vector space, it therefore has a nontrivial kernel. It is easy to see that not every element of $\overline{\mathcal{H}P}$ can have trace $0$, as $\overline{\mathcal{H}P}$ is closed under left multiplication by $i$, $j$, and $k$. The kernel thus has dimension exactly one. It follows that up to scaling, there is a unique nonzero element 
\[ t_P=ai+bj+ck \in \overline{\mathcal{H}P} \]
with trace $0$. But, all elements of $\overline{\mathcal{H}P}$ have norm $0$, so $N(t_P) = 0$, implying
\[ a^2+b^2+c^2=0. \]
Therefore, there is a corresponding point $c_P = (a,b,c) \in C_p$. 

Because $t_P \neq 0$, $\overline{\mathcal{H}}t_P=\overline{\mathcal{H}P}$. Therefore, $t_P$ is a left associate of $P$, and the map $P \mapsto c_P$ is a bijection.
\end{proof}

This bijection lets us reduce the study of metacommutation to an action on $C_p$. 

The following proposition describes the conjugation action on the conic.

\begin{prop}
For a prime $P \in \mathcal{H}$ with norm $p$, let $Q$ be a Hurwitz integer with norm coprime to $p$. If $PQ=Q'P'$, then $\overline{Q^{-1}}t_P\overline{Q}=t_{P'}.$ 
\end{prop}
\begin{proof}
Because $Q$ and $Q'$ are invertible in $\overline{\mathcal{H}}$, \[\overline{Q^{-1}\mathcal{H}PQ}=\overline{\mathcal{H}PQ}=\overline{\mathcal{H}Q'P'}=\overline{\mathcal{H}P'}.\]
As $t_P \in \overline{\mathcal{H}P}$, we can deduce that 
\[ \overline{Q^{-1}}t_P\overline{Q} \in \overline{\mathcal{H}P'}. \] $\overline{Q^{-1}}t_P\overline{Q}$ is nonzero and has trace $0$, so $\overline{Q^{-1}}t_P\overline{Q} = t_{P'}$. 
\end{proof}

Proposition 2.2 gives a simpler way of thinking about the metacommutation action: as a group action by conjugation of the quaternions with non-zero norm in $\overline{\mathcal{H}}$ on the points on our projective conic.

\section{An Isomorphism}
\noindent While the characterization of the metacommutation mapping as a group action on $C_p$, as \cite{cohnkum} gives, is simple, $\overline{\mathcal{H}}$ is still not as intuitive as we want. So in the spirit of simplifying our study of metacommutation, we recall the following proposition. 

\begin{prop}
$\overbar{\mathcal{H}}$ is isomorphic to $\mathcal{M}_2(\mathbb{F}_p)$. 
\end{prop}
\begin{proof}
The isomorphism follows from the fact that $\overbar{\mathcal{H}}$ is a split four-dimensional algebra over $\mathbb{F}_p$, for the sake of clarity we will construct an explicit isomorphism $\varphi:\overbar{\mathcal{H}}\rightarrow\mathcal{M}_2(\mathbb{F}_p)$. If $\gamma \in \overbar{\mathcal{H}}$ and \[ \gamma= \gamma_{1} + \gamma_{2}i + \gamma_{3}j + \gamma_{4}k, \] then
\begin{gather*}
\varphi (\gamma) =  
\begin{pmatrix*}[r]
\gamma_{1} +\gamma_{2}a+\gamma_{4}b & \gamma_{3} + \gamma_{4}a-\gamma_{2}b \\
-\gamma_{3} + \gamma_{4}a-\gamma_{2}b & \gamma_{1}-\gamma_{2}a-\gamma_{4}b 
 \end{pmatrix*}, 
 \end{gather*}
where $a^2+b^2 \equiv -1 \bmod{p}$. (A pigeonhole argument proves such $a$ and $b$  always exist; see \cite{irelandrosen}). Note that under $\varphi$, we have
\[ 1 \mapsto 
\begin{pmatrix}
1 &  0 \\
0 &  1 \end{pmatrix}, \
i  \mapsto \begin{pmatrix*}[r]
a & -b \\
-b  &  -a \end{pmatrix*}, \]
\[ j \mapsto  \begin{pmatrix*}[r]
0 &  1 \\
-1  & 0 \end{pmatrix*}, \
k  \mapsto  \begin{pmatrix*}[r]
b & a \\
a & -b \end{pmatrix*}. \]
Explicit calculation then shows
\[ \varphi(\gamma\delta) = \varphi(\gamma)\varphi(\delta) \quad \text{and} \quad  \varphi(\gamma + \delta) = \varphi(\gamma) + \varphi (\delta). \] (Also, note that 
\[ \varphi(i)^2=\varphi(j)^2=\varphi(k)^2=\varphi(i)\varphi(j)\varphi(k)=\varphi(-1), \]
as desired.) Thus, $\varphi$ is a ring homomorphism, and, since it is bijective, an isomorphism. 
\end{proof}

\begin{corollary} 
$N(\gamma)= \operatorname{det}(\varphi(\gamma))$ and $\operatorname{tr}(\gamma) = \operatorname{tr}(\varphi(\gamma))$, where $\operatorname{tr}(\varphi(\gamma))$ is the standard matrix trace. 
\end{corollary}

\section{An Action on the Projective Line Over $\mathbb{F}_p$}
\noindent Now we characterize the metacommutation action as a group action.

\begin{Metacommutation Action}
The metacommutation action on $C_p$ is isomorphic to the right standard action of $\operatorname{PGL}_2(\mathbb{F}_p)$ on $\mathbb{P}^{1}(\mathbb{F}_p) $, which we use in the form 
\[ \langle x, y \rangle * \begin{pmatrix} a_1 & a_2 \\ a_3 & a_4 \end{pmatrix} = \langle a_1 x + a_3 y, a_2 x + a_4 y \rangle. \]
\end{Metacommutation Action}
\begin{proof}
Let $\mathcal{U}$ denote the elements of $\overbar{\mathcal{H}}$ with non-zero norm. Under the isomorphism $\varphi$, $\mathcal{U} \mapsto \operatorname{GL}_2(\mathbb{F}_p)$. When we projectivize $\operatorname{GL}_2(\mathbb{F}_p)$ to obtain $\operatorname{PGL}_2(\mathbb{F}_p)$, we have an isomorphism 
\begin{align*}
\Theta: \mathcal{U}/\mathbb{F}_p^* \rightarrow \operatorname{PGL}_2(\mathbb{F}_p),
\end{align*}
since $\operatorname{PGL}_2(\mathbb{F}_p) =\operatorname{GL}_2(\mathbb{F}_p)/\operatorname{Z}(\mathbb{F}_p)$. 

Recall that elements of the conic $C_p$ are, by definition, those elements with norm and trace equal zero. Hence, $C_p$ maps to $D$ under $\Theta$, where $D$ denotes the elements in $\mathcal{M}_2(\mathbb{F}_p)$ of the form 
\[ \begin{pmatrix*}[r] -a_1 & a_2 \\ 
a_3 & a_1\end{pmatrix*}, \quad \text{where } a_1^{2}=a_2 a_3. \] 
We can characterize the $p+1$ elements of $D$ as follows.
If $a_3=0$ then $a_1=0$, and therefore only one element of $C_p$ maps to such a matrix: $\begin{psmallmatrix}
0 & a_2 \\
0 & 0
\end{psmallmatrix}$, which is equivalent to $\begin{psmallmatrix}
0 & 1 \\
0 & 0
\end{psmallmatrix}$ in $\operatorname{PGL}_2(\mathbb{F}_p)$. We associate this matrix with the member $\langle 0,1 \rangle$ of the projective line.
Otherwise, by scaling, we can assume $a_3 = 1$; under $\Theta$, the $p$ remaining points of $C_p$ then map to the elements of $D$ of the form
\[ \begin{pmatrix*}[r]
-a_1 & -a_1^2 \\
 1 & a_1
\end{pmatrix*}. \]
We associate these matrices with the elements $\langle 1,a_1 \rangle$ of the projective line. Thus the action of the projectivization of $\overbar{\mathcal{H}}$ on $C_p$ is isomorphic to the conjugation action of $\operatorname{PGL}_2(\mathbb{F}_p)$ on $D$.

Let 
\[ M_0 = \begin{pmatrix} 0 & 1 \\ 0 & 0 \end{pmatrix} \in D \quad \text{and} \quad A = \begin{pmatrix} a_1 & a_2 \\ a_3 & a_4 \end{pmatrix} \in \operatorname{PGL}_2 (\mathbb{F}_p). \]
Then,
\begin{IEEEeqnarray*}{rCl}
A^{-1} M_0 A &=& 
\begin{pmatrix}
a_1 & a_2 \\ 
a_3 & a_4
\end{pmatrix}^{-1}
\begin{pmatrix}
 0 & 1 \\
 0 & 0
\end{pmatrix}
\begin{pmatrix}
a_1 & a_2 \\ 
a_3 & a_4
\end{pmatrix} \\
&=&
\begin{pmatrix*}
a_4 a_3 & {a_4}^{2} \\
-{a_3}^2 & -a_4 a_3 
\end{pmatrix*}. \IEEEyesnumber 
\end{IEEEeqnarray*}
It is easy to see that this matrix corresponds to the element $\langle a_3,a_4 \rangle$ of the projective line.
This agrees with the standard action of $\operatorname{PGL}_2(\mathbb{F}_p)$, as

\begin{IEEEeqnarray*}{rCl}
\langle 0,1 \rangle 
\begin{pmatrix}
a_1 & a_2 \\
a_3 & a_4
\end{pmatrix}
&=& \langle a_3,a_4 \rangle
\end{IEEEeqnarray*}

Therefore the action by conjugation of $\operatorname{PGL}_2 (\mathbb{F}_p)$ on the element of $D$ of the form $\begin{psmallmatrix}
 0 & 1 \\
 0 & 0
\end{psmallmatrix}$ is isomorphic to the right standard action of $\operatorname{PGL}_2 (\mathbb{F}_p)$ on $\mathbb{P}^1 (\mathbb{F}_p)$.

Next, we prove the same result for elements $M \in D$ of the form $\begin{psmallmatrix}
-m & -m^2 \\
1 & m
\end{psmallmatrix}$. If $A \in \operatorname{PGL}_2 ( \mathbb{F}_p )$, then
\begin{IEEEeqnarray*}{rCl}
A^{-1}MA &=& \begin{pmatrix*}[l]
a_1 & a_2 \\ 
a_3 & a_4
\end{pmatrix*}^{-1}
\begin{pmatrix}
-m & -m^2 \\
1 & m
\end{pmatrix}
\begin{pmatrix*}[r]
a_1 & a_2 \\ 
a_3 & a_4
\end{pmatrix*}
\\
&=& \begin{pmatrix*}[r]
a_4 & -a_2 \\ 
-a_3 & a_1
\end{pmatrix*}
\begin{pmatrix}
-m & -m^2 \\
1 & m
\end{pmatrix}
\begin{pmatrix*}[r]
a_1 & a_2 \\ 
a_3 & a_4
\end{pmatrix*} 
\\
&=&
\begin{pmatrix}
-(a_1 + a_3 m) (a_2 + a_4 m) & -(a_2 + a_4 m)^2 \\
(a_1 + a_3 m)^2 & (a_1 + a_3 m) (a_2 + a_4 m) 
\end{pmatrix}. \IEEEyesnumber
\end{IEEEeqnarray*}
It is easy to see that this corresponds to the element of $\langle a_1+a_3m,a_2+a_4m \rangle$ of the projective line.
This again agrees with the standard action of $\operatorname{PGL}_2(\mathbb{F}_p)$, as

\begin{IEEEeqnarray*}{rCl}
\langle 1,m \rangle 
\begin{pmatrix}
a_1 & a_2 \\
a_3 & a_4
\end{pmatrix}
&=& \langle a_1+a_3m,a_2+a_4m \rangle
\end{IEEEeqnarray*}

We have proven the metacommutation action on $C_p$ is isomorphic to the action of $\operatorname{PGL}_2 (\mathbb{F}_p)$ on $D$ by conjugation, and that the action of $\operatorname{PGL}_2 (\mathbb{F}_p)$ on $D$ by conjugation is isomorphic to the right standard action of $\operatorname{PGL}_2  (\mathbb{F}_p)$ on $\mathbb{P}^1 (\mathbb{F}_p)$. Therefore the metacommutation action on $C_p$ is isomorphic to the right standard action of $\operatorname{PGL}_2 (\mathbb{F}_p)$ on $\mathbb{P}^1 (\mathbb{F}_p)$.
\end{proof}

\section{Main Results}
\noindent We now present new, shorter proofs of some previously known results about the metacommutation map, as well as new information our isomorphism provides.
In the following calculations, $p$ is an odd prime, and $Q \in \mathcal{H}$ has prime norm $q \neq p$.

\begin{Sign of Permutation} [Cohn and Kumar, 2013] 
The sign of the metacommutation map of a Hurwitz prime of norm $q$ on the Hurwitz primes of norm $p$ is $\genfrac(){}{}{q}{p}$.
\end{Sign of Permutation}
\begin{proof}
Recall three facts:
\begin{enumerate}
\item The determinant of a matrix equals the norm of its associated quaternion.
\item Due to the equivalence of group actions, the sign of the standard action of the associated matrix on the elements of $\mathbb{P}^1(\mathbb{F}_p)$ equals the sign of the metacommutation map of the quaternion. 
\item The sign of an element $A$ of $\operatorname{GL}_2(\mathbb{F}_p)$ is the quadratic character of its determinant modulo $p$; that is, the Legendre symbol \[ \left(\frac{\operatorname{det}(A)}{p} \right). \]
\end{enumerate}
Sign does not vary with multiplication by a scalar matrix, because the determinant of a scalar $2$ by $2$ matrix is a square, and the determinant is multiplicative; thus, the sign of a matrix in $\operatorname{PGL}_2(\mathbb{F}_p)$ is well-defined.
It is well known that if the sign of the permutation induced by an element of $\operatorname{PGL}_2(\mathbb{F}_p)$ on $\mathbb{P}^1(\mathbb{F}_p)$ is $1$, then the sign of the matrix of the element is also $1$, and the element is therefore in $\operatorname{PSL}_2(\mathbb{F}_p)$ by definition. This implies its determinant is a square (because we can think of $\operatorname{PSL}_2(\mathbb{F}_p)$ as the subgroup of   $\operatorname{PGL}_2(\mathbb{F}_p)$ containing all matrices of square determinant). Similarly, if the sign of the metacommutation map is $-1$, then the determinant of the associated matrix is not a square. But the determinant of the representation of $Q$ is equal to the norm of $Q$, so the sign of the permutation equals $\genfrac(){}{}{q}{p}$.
\end{proof}
\begin{Fixed Points}[Cohn and Kumar, 2013]The number of fixed points of the metacommutation map on $p$ is 
\[ 1+ \left(\frac{\operatorname{tr}(Q)^2-4q}{p}\right), \] 
unless $\overline{Q} \in \mathbb{F}_p$, in which case every point is a fixed point.
\end{Fixed Points}
\begin{proof}
Let 
\[ M_Q=
 \begin{pmatrix}
a_1 & a_2 \\
a_3 & a_4 \end{pmatrix} \]
be the element of $\operatorname{PGL}_2 (\mathbb{F}_p)$ associated to $\overline{Q}$. The characteristic polynomial of $M_Q$ is \[ (a_1-x)(a_4-x)-a_2a_3=x^2-\operatorname{tr}(Q)x+q. \]
The number of distinct roots in $\mathbb{F}_p$ of this polynomial is 
  \[ 1+\left(\frac{\operatorname{tr}(Q)^2-4q}{p}\right). \]
This number of eigenvalues is the same as the number of fixed points of the projectivized map, with the exception of when there is an eigenvalue with geometric multiplicity of $2$, in which case all points will be fixed points. This occurs when $M_Q$ is diagonal and so $\overline{Q}$ is in $\mathbb{F}_p$.
\end{proof}

We can also study the cycle structure of the metacommutation maps using through $\operatorname{PGL}_2(\mathbb{F}_p)$, thanks to Theorem 4.1.
\begin{Cycles} All of the cycles which are not fixed points in a metacommutation map have the same length.
\end{Cycles}
\begin{proof}
We proceed by contradiction. Suppose a permutation $M_Q \in \operatorname{PGL}_2(\mathbb{F}_p)$ of the points of $\mathbb{P}^1 (\mathbb{F}_p)$ had cycles with lengths $m$ and $n$ where $m>n>1$. Then $(M_Q)^{n}$ would have at least $n$ fixed points, but would not be the identity permutation. Thus $n=2$, because (by Theorem 5.2) permutations with more than two fixed points fix all $p+1$ points of $\mathbb{P}^1 (\mathbb{F}_p)$. We now show that if such an $M_Q$ contains a cycle of length 2, then $M_Q$ itself must have order 2, and all cycles of $M_Q$ have length $1$ or $2$.

Suppose for some transformation
\[ M_Q =\begin{pmatrix}
a_1 & a_2 \\
a_3 & a_4 \end{pmatrix} \in \operatorname{PGL}_2 (\mathbb{F}_p) \] that
$\langle
x, y  \rangle \in \mathbb{P}^1 (\mathbb{F}_p)$ has order $2$, so that
\[ \langle
x, y \rangle 
\begin{pmatrix}
a_1 & a_2 \\
a_3 & a_4 \end{pmatrix}^2
=
 \langle
x, y \rangle. \]
Then,
\begin{IEEEeqnarray*}{rCl}
\langle x, y \rangle
\begin{pmatrix}
a_1 & a_2 \\
a_3 & a_4 
\end{pmatrix}^2
&=& \langle a_1x+a_3y, a_2x+a_4y  \rangle
 \begin{pmatrix}
a_1 & a_2 \\
a_3 & a_4 \end{pmatrix} \\ 
&=&
\big \langle (a_1^2+a_2a_3)x+(a_1a_3+a_3a_4)y, \IEEEyesnumber \\
&& \> (a_2a_1+a_2a_4)x+(a_4^2+a_2a_3)y \big \rangle.
\end{IEEEeqnarray*}
Hence, for some non-zero $\lambda \in \mathbb{F}_p$,
\begin{gather*} \lambda x=(a_1^2+a_2a_3)x+a_3(a_1+a_4)y, \\ \lambda y = a_2(a_1+a_4)x+(a_4^2+a_2a_3)y. \end{gather*} 

Suppose $a_1+a_4=0$. Then, $ \lambda x=(a_1^2+a_2a_3)x$ and 
\[ \lambda y=(a_4^2+a_2a_3)y=(a_1^2+a_2a_3)y. \]
Note that $x$ and $y$ cannot both be zero, so $a_1^2+a_2a_3=\lambda$. Hence, all points of $\mathbb{P}^1 (\mathbb{F}_p)$ are fixed points under $M_Q^2$ and all cycles will have length $1$ or $2$. Otherwise, 
\begin{equation}
(a_1^2+a_2a_3)xy+a_3(a_1+a_4)y^2=(a_4^2+a_2a_3)xy+a_2(a_1+a_4)x^2 . 
\end{equation}
Dividing (5.2) by $a_1+a_4$ yields
\[ (a_1-a_4)(xy)=a_2x^2-a_3y^2. \]
This implies
\[ x(a_2x+a_4y)=y(a_1x+a_3y), \]
and thus
\begin{gather}
\langle x , y \rangle \sim \langle a_1x+a_3y, a_2x+a_4y \rangle, 
\end{gather}
so
\begin{gather*}
\langle x, y \rangle 
\begin{pmatrix}
a_1 & a_2 \\
a_3 & a_4 
\end{pmatrix} 
= \langle
x, y \rangle, 
\end{gather*}
which means $\langle x, y \rangle$ has order 1, a contradiction. Hence, all cycles have orders 1 and 2. This proves all the cycles of an element of $\operatorname{PGL}_2(\mathbb{F}_p)$ have the same length. 
\end{proof}

Theorem 5.3 immediately implies the cycle length divides $p+1$, $p$, or $p-1$ depending on whether the permutation has respectively $0$, $1$, or $2$ fixed points. In fact, the number of permutations (and thus the number of distinct metacommutation maps) with order $k>1$ is well-known to equal
\[ \begin{cases}
\varphi(k) p (p - 1) / 2 & \text{if } k \mid (p + 1), \\
\varphi(k) p (p + 1) / 2 & \text{if } k \mid (p - 1), \\
p^2 - 1 & \text{if } k = p.
\end{cases} \]
Notably, the elements of the projective linear group have a nice presentation:
\[ \operatorname{PGL}(2, p) = \langle a, b, \mid a^2b^p= (ab^2)^4=(abab^2)^3b^p=1  \rangle, \]
which \cite{robwil} gives.

\section{A Generalization}
\noindent The results in Section 5 are stated in terms of the ring of Hurwitz quaternions. However, the only specific properties needed of an order $\mathcal{O}$ over the quaternions for these results to extend are a division algorithm for primes of norm $p$ and an isomorphism of $\mathcal{O}/p\mathcal{O}$ with $M_2(\mathbb{F}_p)$. This latter property is equivalent to $\mathcal{O} \otimes_{\mathbb{Z}} \mathbb{F}_p$ being nontrivial.  In particular, the Lipschitz quaternions
\[ \mathcal{L} = \{ a + bi + cj + dk\ |\ a, b, c, d \in \mathbb{Z}\}, \]
do not generally have a division algorithm, which is why results about metacommutation were first stated in terms of the Hurwitz integers. But in the Lipschitz integers, we can divide $Q$ by $P$ as long the norm of $P$ is an odd rational prime, and so our results on metacommutation still hold.

\section{Acknowledgements}
\noindent The authors thank Henry Cohn for suggesting this area of research and for providing thoughtful comments; likewise, we thank Raffael Singer for his invaluable guidance during the period when this research was conducted. We also thank the Clay Mathematics Institute for funding the CMI/PROMYS research labs, as well as Erick Knight, Glenn Stevens, and the PROMYS program for making this research possible. Lastly we would like to thank Tara Smith and Daniel Allcock for helpful comments, and Edward Sanger for his support.

\end{document}